\documentclass[12pt]{amsart}
\usepackage{amssymb,amsmath,latexsym,amscd,amsfonts}

  \makeatletter
  \CheckCommand*\refstepcounter[1]{\stepcounter{#1}%
      \protected@edef\@currentlabel
       {\csname p@#1\endcsname\csname the#1\endcsname}%
  }
  \renewcommand*\refstepcounter[1]{\stepcounter{#1}%
    \protected@edef\@currentlabel
      {\csname p@#1\expandafter\endcsname\csname the#1\endcsname}%
  }
  \def\labelformat#1{\expandafter\def\csname p@#1\endcsname##1}
  \DeclareRobustCommand\Ref[1]{\protected@edef\@tempa{\ref{#1}}%
     \expandafter\MakeUppercase\@tempa
  }
  \makeatother

  \makeatletter
  \newcommand{\numberlike}[2]{%
     \expandafter\def\csname c@#1\endcsname{%
         \expandafter\csname c@#2\endcsname}%
  }
  \makeatother

 \def\DefaultNumberTheoremWithin{section}

  \theoremstyle{plain}
  \newtheorem{lem}{Lemma}
     \numberwithin{lem}{\DefaultNumberTheoremWithin}
     \labelformat{lem}{Lemma~#1}
  
     \numberwithin{Claim}{\DefaultNumberTheoremWithin}
     \numberlike{Claim}{lem}
     \labelformat{Claim}{Claim~#1}
  \newtheorem{thm}{Theorem}
     \numberwithin{thm}{\DefaultNumberTheoremWithin}
     \numberlike{thm}{lem}
     \labelformat{thm}{Theorem~#1}
  
     \numberwithin{cor}{\DefaultNumberTheoremWithin}
     \numberlike{cor}{lem}
     \labelformat{cor}{Corollary~#1}
  \newtheorem{prop}{Proposition}
     \numberwithin{prop}{\DefaultNumberTheoremWithin}
     \numberlike{prop}{lem}
     \labelformat{prop}{Proposition~#1}

  \newtheorem{Conjecture}{Conjecture}
     \numberwithin{Conjecture}{\DefaultNumberTheoremWithin}
     \numberlike{Conjecture}{lem}
     \labelformat{Conjecture}{Conjecture~#1}

  \theoremstyle{definition}
  \newtheorem{Definition}{Definition}
     \numberwithin{Definition}{\DefaultNumberTheoremWithin}
     \numberlike{Definition}{lem}
     \labelformat{Definition}{Definition~#1}

  \theoremstyle{definition}
  
     \numberwithin{Question}{\DefaultNumberTheoremWithin}
     \numberlike{Question}{lem}
     \labelformat{Question}{Question~#1}

  \theoremstyle{definition}
  
     \numberwithin{Problem}{\DefaultNumberTheoremWithin}
     \numberlike{Problem}{lem}
     \labelformat{Problem}{Problem~#1}

  \theoremstyle{remark}
  
     \numberwithin{Remark}{\DefaultNumberTheoremWithin}
     \numberlike{Remark}{lem}
     \labelformat{Remark}{Remark~#1}

     \numberwithin{Example}{\DefaultNumberTheoremWithin}
     \numberlike{Example}{lem}
     \labelformat{Example}{Example~#1}
  
     \labelformat{Case}{Case~#1}
     \numberwithin{Case}{lem}
  
     \labelformat{Step}{Step~#1}
     \numberwithin{Step}{lem}

  \labelformat{equation}{(#1)}
  \labelformat{figure}{Figure~#1}
  \labelformat{section}{\S#1}
  \labelformat{subsection}{\S#1}
  \labelformat{subsubsection}{\S#1}

  \def\eqref{\ref}

\def\hpic #1 #2 {\mbox{$\begin{array}[c]{l} \epsfig{file=#1,height=#2} \end{array}$}}
\def\vpic #1 #2 {\mbox{$\begin{array}[c]{l} \epsfig{file=#1,width=#2} \end{array}$}}
\def\ignore #1 {}

\def\KK{\mathbb{K}}
\def\FF{\mathbb{F}}

\def\ZZ{\mbox{{$\mathbb{Z}$}}}
\def\NN{\mbox{{$\mathbb{N}$}}}

\def\B{\mbox{${\mathcal B}$}}

\def\K{\mathcal{K}}
\def\Cl{\mathrm{Cl}}

\def\Ho{\tilde{\mathrm{H}}}
\def\Pr{\mathbb{P}}

\newcommand{\Tor}{{\mathrm{Tor}}}
\def \aaron #1 {\marginpar{#1 -AA}}
\newcommand{\pdim}{{\mathrm{pdim}}}

\begin{document}

\title{}
\author[E. Babson]{Eric Babson}
\address{Department of Mathematics, UC Davis, One Shields Ave, Davis, CA 95616, USA}
\email{babson@math.usdavis.edu}
\author[V. Welker]{Volkmar Welker}
\address{Philipps-Universit\"at Marburg \\
Fachbereich Mathematik und Informatik \\
35032 Marburg \\ Germany}
\email{welker@mathematik.uni-marburg.de}
\title[Higher dimensional connectivity and minimal degree]{Higher dimensional connectivity and minimal degree of 
random graphs with an eye towards minimal free resolutions}
\thanks{This material is based upon work supported by the National Science Foundation under Grant No. DMS-1440140 while the second author was in 
residence at the
Mathematical Sciences Research Institute in Berkeley, California, USA}
\maketitle

\section{Introduction}
\label{sec:intro}

In this note we define and study graph invariants generalizing to higher 
dimension the maximum degree of a vertex and the vertex-connectivity 
(our $0$-dimensional cases). These are known to coincide almost surely 
in any regime for Erd\"os-R\'enyi random graphs. We show the same in the 
one dimensional case for a middle density regime and show the easier 
inequality for all dimensions in the same regime.

Our original motivation comes from the study of minimal free resolutions in 
commutative algebra and is explained in \ref{sec:motivation}.

Let us fix some notation.
Consider a (finite simple) graph $\Gamma = 
(V (\Gamma), E(\Gamma))$ with vertices $V = V (\Gamma)$ and edges
$E = E(\Gamma) \subseteq \binom{V}{2}$.  If 
$C \subseteq V$ write 
$\Gamma|_{V \setminus C} = \Gamma - C = (V \setminus C, E \cap \binom{V \setminus C}{2})$
for the subgraph of 
$\Gamma$ induced on $V \setminus C$. If $v\in C\subseteq V(\Gamma)$ write  
$N_\Gamma(v) = \{u \in V (\Gamma)~:~ \{u, v\} \in E(\Gamma)\}$ for the neighbors of $v$ and 
$N_\Gamma (C) = \bigcap_{c\in C} N_\Gamma (c)$ for the
common neighbors of $C$. We denote by $\Cl(\Gamma)\subseteq 2^{V(\Gamma)}$ the clique complex of $\Gamma$ which is the simplicial complex on ground set $V$ whose $i$-simplices $\Cl(\Gamma)[i]=\Cl(\Gamma)\cap \binom{V}{i+1}$ are
the subsets $C \in \binom{V}{i+1}$ such that
$\binom{C}{2} \subseteq E(\Gamma)$.

\begin{Definition}
  If $\Gamma$ is a finite simple graph, $\KK$ a field and $i$ a natural number write 
  $$\kappa^i_\KK(\Gamma)=\min\Big\{|C| ~:~C \subseteq V(\Gamma),\, \Ho^i\big(\Cl(\Gamma-C);\KK\big)\neq 0\Big\}$$  
  for the $i$-dimensional homological connectivity of $\Gamma$
  and
  $$\delta^i(\Gamma)=\min\Big\{ |N_\Gamma(I)|~ :~ I \in \Cl(\Gamma)[i]\Big\}$$ for the maximum degree of an $i$-dimensional simplex in  $\Cl(\Gamma)$. 
  \end{Definition}

As usual in the definition 
the minimum over the empty set is set to be $+\infty$.
Thus $\kappa^0(\Gamma)=\kappa^0_\KK(\Gamma)$ is the minimum number of vertices whose removal disconnects $\Gamma$ (hence independent of $\KK$) and $\delta^0(\Gamma)$ is the minimum degree of a vertex which (except for the complete graph) is also the minimum number of vertices whose removal disconnects $\Gamma$ leaving a singleton component.  

For a finite set $V$ and a probability $p\in[0,1]$ write $\Omega^V_p$ for the Erd\"os-R\'enyi probability model of
random graphs with vertex set $V$
and edges chosen independently with probability $p$.
For a function
$p : \NN \rightarrow [0,1]$ assigning probabilities to each
natural number $\NN = \{1,2,\ldots \}$ we write $p_n$ for $p(n)$.  
In most of our applications we will consider the
Erd\"os-R\'enyi model on vertex set $V = [n] := \{
1,\ldots, n\}$ for some $n \in \NN$. 

\begin{thm}[Bollob\'as, Thomasson \cite{BT}]
  \label{classical}
  If $p:\NN\rightarrow [0,1]$ then
  $$\lim_{n \rightarrow \infty} \Pr_{\Gamma \in \Omega^{[n]}_{p_n}}\Big(\kappa^0(\Gamma) = \delta^0(\Gamma)\Big) = 1.$$
\end{thm}

Note that $\kappa^1_\KK(\Gamma)$ is the minimum number of vertices whose removal leaves a fundamental group with nontrivial maps to $\ZZ\slash p$ with $p$ the characteristic of $\KK$.  This will often agree with maps to $\ZZ$ but will occasionally depend on $p$ as for instance if $T=\bar H^1(\Cl(\Gamma),\ZZ)$ is nontrivial torsion so that $\kappa^1_{\KK}(\Gamma)\not=0$ iff the characteristic of $\KK$ divides the order of $T$.  

Note that $\delta^1(\Gamma)$ is the minimum number of vertices whose removal leaves an edge as a maximal face in the clique complex.  If the resulting maximal edge is not an isthmus (its ends are connected by a path not using that edge) this gives a $\ZZ$ summand in the first homology.  In many regimes this is asymptotically almost surely true and is the $1$-dimensional analog of not being a complete graph.  

\begin{Definition} Call $p : \NN \rightarrow [0,1]$ middling of exponent $w >0$ if
  $$ -\liminf_{n \rightarrow \infty}\log_np_n< w^{-1}$$ and
  $$\lim_{n \rightarrow \infty} p_n=0. $$
\end{Definition}

\begin{thm} 
	\label{thm:d1=k1}
	If $p : \NN \rightarrow [0,1]$ is middling of exponent $10$ then
	$$\lim_{n\rightarrow\infty}\Pr_{\Gamma\in\Omega^{[n]}_{p_n}}\ \Big(\kappa^1_{\FF_2}(\Gamma)=\delta^1(\Gamma)\Big)=1.$$  
\end{thm}

We view $\kappa^i_\KK(\Gamma)$ as an $i$-dimensional version of vertex connectivity. To explain that we need some more notation. 

Let $\KK$ be a field.
If $I\in \Cl(\Gamma)[i]$
(i.e. $I\subseteq V(\Gamma)$, $|I|=i+1$ and $\binom{I}{2}  \subseteq E(\Gamma)$) 
we write $\chi_{\{I\}}$ (resp. $\zeta_{\{I\}}$) for its characteristic
$i$-cochain (resp. $i$-chain).
These form a basis of the cochain group $C^i(\Cl(\Gamma);\KK)$
(resp. the chain group $C_i(\Cl(\Gamma);\KK)$).
Analogously, for a collection $E$ of $i$-cliques in $\Gamma$ we write $\chi_E$ 
(resp. $\zeta_E$) for the $i$-cochain 
$\sum_{I \in E} \chi_{\{I\}}$ (resp. $i$-chain $\sum_{I \in E} \zeta_{\{I\}}$). 
In the case we eventually focus on in which $\KK=\FF_2$ these are all the
$i$-cochains (resp. $i$-chains).  
For any cocycle $\gamma$ in the cocycle group $Z^i(\Cl(\Gamma);\KK)$
(resp. cycle group $Z_i(\Cl(\Gamma);\KK)$) we
denote by $[\gamma]$ its class in $\Ho^i(\Cl(\Gamma);\KK)$
(resp. $\Ho_i(\Cl(\Gamma);\KK)$). 

Now if $I\in\Cl(\Gamma)[i]$ is any $(i+1)$-clique then deleting all vertices
$C=N_\Gamma(I)$ connected to $I$ yields an integer $i$-cocycle
$\gamma=\chi_{\{I\}}\in Z^i(\Cl(\Gamma-C);\ZZ)$.  In the $i=0$ case
$[\gamma]\in\Ho^i(\Cl(\Gamma-C);\ZZ)$ is non-zero unless $C\cup I=V(\Gamma)$.
It follows that $\kappa_\KK^0 (\Gamma)\leq \delta^0(\Gamma)$ unless 
$\Gamma$ is a complete graph. In the $i=1$ case $[\gamma]$ is non-zero
unless $I$ is an isthmus in $\Gamma-C$.  
For $i \geq 2$ the non-triviality condition is slightly more complicated. 
We show that the inequality $\kappa_{\FF_2}^i(\Gamma) \leq \delta^i(\Gamma)$
holds asymptotically almost surely in the middle density regime.  

\begin{prop}
    \label{lem:ki<=di}
    If $i \geq 0$ is an integer and $p : \NN \rightarrow [0,1]$ is middling
    of exponent $2i$ then
	$$\lim_{n\rightarrow\infty}\Pr_{\Gamma\in\Omega^{[n]}_{p_n}}\ \Big(\kappa^i_{\FF_2}(\Gamma)\leq\delta^i(\Gamma)\Big)=1.$$  
\end{prop}

We indeed we conjecture that at least in the middling regime the
analog of \ref{thm:d1=k1} holds for all $i$.

\begin{Conjecture}
    \label{con:main}
    For every integer $i$ there is a $w_i > 0$ so that if
    $p : \NN \rightarrow [0,1]$ is middling of exponent $w_i$ then 
	$$\lim_{n\rightarrow\infty}\Pr_{\Gamma\in\Omega^{[n]}_{p_n}}\ \Big(\kappa^i_{\FF_2}(\Gamma)=\delta^i(\Gamma)\Big)=1.$$
\end{Conjecture}

\section{Motivation for the invariants}
\label{sec:motivation}

Our original motivation for studying $\kappa^i_\KK(\Gamma)$ comes from commutative algebra. 
For a finite simplicial complex $\Delta$ 
let $S = \KK[\{x_v\}_{v\in \Delta[0]}]$. The Stanley-Reisner ring of $\Delta$ is the $\NN^{\Delta[0]}$-graded $S$-algebra  $\KK[\Delta]=S\slash I$ where $I=(\,x^U=\Pi_{u\in U}x_u\ |\ U\not\in \Delta\,)$.  

One of the most important invariants of finitely generated graded $S$-modules $M$ is the graded Betti table $$\{\beta_{i,i+j}(M)= \dim_\KK \Tor_i^S(M;\KK)_{i+j}\}_{i,j}$$ or equivalently the $S$-ranks arising in any minimal graded free $S$-module resolution of $M$.  In the Stanley-Reisner case there is a cohomological interpretation given by Hochster's formula which connects them to our invariants:

\begin{eqnarray}
  \label{HochsterFormula}
  \beta_{i,i+j}(\KK[\Delta]) & = & \sum_{W \in \binom{\Delta[0]}{i+j}} \dim_\KK \Ho^{j-1}(\Delta|_W;\KK).
\end{eqnarray}

Here $\Delta|_W$ is the induced subcomplex of $\Delta$ with vertices $\Delta|_W[0]=W$.
Note that except for $\beta_{0,0}(\KK[\Delta])=1$ we have $\beta_{i,\ell} (\KK[\Delta]) = 0$ if $\ell\leq i$ or $i\leq 0$.
Boij-S\"oderberg theory has improved our understanding of the possibilities for these Betti tables and it is understood that the restriction to monomial ideals or Stanley-Reisner modules reduces the possibilities but little is understood about what those restrictions are.  
The $i$\textsuperscript{th} row $\{\,\beta_{j,i+j}(M)\,\}_{j \geq 0}$ is called the $i$\textsuperscript{th} strand of the resolution and its length $\lambda^i(M)=\max\{ j~:~\beta_{j,i+j}(M) \neq 0\}$ is an important aspect of the Betti table.  In the Stanley-Reisner case
\ref{HochsterFormula} implies that $$\lambda^i(\KK[\Cl(\Gamma)])=n-\kappa^{i-1}_\KK(\Gamma).$$

Since the determination of the length of a strand through $\kappa^i_\KK(\Gamma)$ involves a 
subtle study of the
local structure of a simplicial complex it is 
not so simple to approach. On the other hand $\delta^i(\Gamma)$ is a much simpler graph
invariant that is more easily amenable to the
classical tools of (random) graph theory. Hence a positive answer to \ref{con:main} would provide a powerful tool for the
probabilistic analysis of strand lengths.

In our analysis of the clique complex some homological measures of complexity come up (see \ref{con:delicate}), 
which translate into 
measures of complexity of modules over the polynomial ring.
We think that these measures have the potential to be of 
independent interest. For that reason, we defines them in the
next paragraph and provide the argument why they can be
seen as generalizations of the homological measures on 
simplicial complexes from \ref{cohomnorm} and \ref{homnorm}.

We set $S=\KK[x_1,\ldots, x_n]$ and 
consider a subset $U \subseteq \binom{\{x_1,\ldots, x_n\}}{i+1}$ as a set of squarefree monomials of degree $i+1$. Then for a graded $S$-module $M$ and $i \geq 1$ 
we set 

\begin{align*} 
\|M\|^i=\min\Big\{\,r &\,:  \,\exists{U\subseteq\binom{\{x_1,\ldots, x_n\}}{i+1},|U|=r}, \\ & \Tor^S_{n-i}(M,\KK)_n\rightarrow\Tor^S_{n-i}(M/(U),\KK)\hbox{ is not injective}\Big\}
\end{align*}

and

\begin{align*}
    \|M\|_i=\min\Big\{\ell\,:\,\Tor^S_{n-i} & (M,\KK)_n\rightarrow\bigoplus_{\genfrac{}{}{0pt}{}{U\subseteq\binom{\{x_1,\ldots, x_n\}}{i+1}}{|U|\geq n-\ell}}\Tor^S_{n-i}(M/(U),\KK) \\ & \hbox{ is injective}\Big\}.
\end{align*}

Consider the case $M = \FF_2[\Cl(\Gamma)]$ is the Stanley-Reisner ring of a clique complex of a graph $\Gamma$.
Using \ref{HochsterFormula} we translate $\|M\|^i$ and
$\|M\|_i$ back into cohomological and 
homological invariants of the complex. 
Let $\Gamma$ be a graph with $V(\Gamma) = [n]$. Then by
\ref{HochsterFormula} we have 
the following $\KK$-vectorspace isomorphism
\begin{align}
    \label{eqhom1}
\Tor^S_{n-i} (\KK[\Cl(\Gamma)],\KK)_n
&\cong& \bar H^{i}(\Cl(\Gamma);\KK)
\end{align}
and
\begin{align}
\label{eqhom2}
\Tor^S_{n-i}(\KK[\Cl(\Gamma)]/(U),\KK) &\cong&
\bigoplus_{j \geq 0} \bigoplus_{W \in \binom{[n]}{n-i+j}} \bar H^{j-1}((\Cl(\Gamma)^U)|_W;\KK),
\end{align}
where $\Cl(\Gamma)^U$ is the simplicial complex with all
faces containing a set supporting a monomial in $U$ removed.
On the side of simplicial complex side the map from 
the right hand side of \eqref{eqhom1} to \eqref{eqhom2} is
induced by the inclusion of simplicial complexes. 
In case $\KK= \FF_2$ assume that $\zeta$ is an $i$-cocycle for $\Cl(\Gamma)$ which represents a non-trivial class and has minimal support in 
its class. Then its image in the right hand side of \eqref{eqhom2} is zero 
if and only if $U$ contains all
$i$-simplices from the support of $\zeta$.
Thus $||\Cl(\Gamma)||^i$ is the minimal supports of a non-trivial $i$-cocycle. We will encounter this invariant 
again for $\KK = \FF_2$ in \ref{cohomnorm} as $\| \bar H^i(\Cl(\Gamma);\FF_2)\|^c$.

Now consider $\|M\|_i$ for $M = \FF_2[\Cl(\Gamma)]$. 
Again using \eqref{eqhom1} and \eqref{eqhom2} and duality between homology and cohomology we see that 
the map from the definition of $\| \FF_2[\Cl(\Gamma)] \|_i$ 
is injective as long as there is no homological basis of
$i$-cycles that is destroyed by the removal any set of 
$\binom{n}{i+1} - \ell$ simplices of dimension $i$. 
of $\Gamma$. This is the case if for all bases there is 
a cycle supported on $\geq \ell$ elements. From this it 
follows that $\| \FF_2[\Cl(\Gamma)]\|_i$ is the 
invariant $\| \bar H_i(\Cl(\Gamma);\FF_2) \|$ from 
\ref{homnorm}.


Finally we would like to mention some related work. 
In \cite{EY} the quotient $\frac{\lambda^i(\KK[\Cl(\Gamma)])}{\pdim(\KK[\Cl(\Gamma)])}$  has been considered in the Erd\"os-R\'enyi model. 
Here $\pdim \KK[\Cl(\Gamma)] = \max_i \lambda^i(\KK[\Cl(\Gamma)])$ is the projective dimension of
$\KK[\Cl(\Gamma)]$. The strongest results from \cite{EY} on that quotient are also for
the middling probability regime but their results do not follow from or imply ours.

In \cite{DPSSW} and \cite{DHKS} among others one can also find results on the projective dimension of random 
monomial ideals in other models.

\section{Measure concentration}

In studying middle density random graphs the uniformity in vertex degrees and in connectivity between large subsets of vertices will play the major role.  

For $k \in \NN$ and a graph $\Gamma$ we denote by
$$D_k(\Gamma)  := \Big\{|N_\Gamma(A)|\ :\ A \subseteq V(\Gamma), |A|= k\Big\}$$
the set of neighborhood sizes of cardinality $k$ vertex sets.  We write $d^-_k(\Gamma)$ for the minimum and $d^+_k(\Gamma)$ for the maximum of $D_k(\Gamma)$. 
Further for $a,b \geq 0$ and a graph $\Gamma$ we denote by 
$$B_{a,b}(\Gamma)  := \Big\{ |E(K_{A,B}) \cap E(\Gamma)|\ :\
     A,B \subseteq V(\Gamma), \,\,\genfrac{}{}{0pt}{}{A \cap B = \emptyset}{|A| \geq a, |B| \geq b} \,\,\Big\}$$
the set of interconnection sizes and write
$b_{a,b}(\Gamma)$ for the minimum of $B_{a,b}(\Gamma)$.  

\begin{prop} 
  \label{prop:9}
  If $p:\NN\rightarrow [0,1]$ is middling of exponent $w$ and $\epsilon>0$ then
  \begin{enumerate}
    \item[(i)] For any integer $k\leq w$ we have
  $$\lim_{n\rightarrow \infty}\Pr_{\Gamma \in \Omega^{[n]}_{p_n}}\Big(n\,p_n^k(1- \epsilon)\leq d^-_k(\Gamma)\leq d^+_k(\Gamma)\leq n\,p_n^k(1+ \epsilon)\Big) =1.$$ 
    \item[(ii)] For any positive $\alpha_0$, $\alpha_1$, $a'$, $\beta_0$, $\beta_1$ and $b'$ with $\alpha_0+\beta_0>1$ and $\frac{\alpha_1+\beta_1+1}{\alpha_0+\beta_0-1}\leq w$ write $a_n=n^{\alpha_0}p_n^{\alpha_1}a'$ and $b_n=n^{\beta_0}p_n^{\beta_1}b'$. Then we have 
 $$ \lim_{n\rightarrow \infty}\Pr_{\Gamma \in \Omega^{[n]}_{p_n}} \Big(b_{a_n,b_n}(\Gamma) \geq a_nb_np_n(1- \epsilon)\Big)  =1.$$

  \end{enumerate}
\end{prop}

For the proof of \ref{prop:9} we use well known results on
tails. Our formulation of the Chernoff-Hoeffding bound
is an immediate consequence of \cite[Corollary 21.7]{FK}.

\begin{thm}[Chernoff-Hoeffding]
    \label{thm:CH}
    If $(\Sigma,\B,\nu)$ is a probability space, $X:\Sigma\rightarrow [0,1]^m$ is measurable with $\nu$-independent coordinate projections, $s:[0,1]^m\rightarrow[0,m]$ is the coordinate sum, $\mu$ the $\nu$-mean of $\sigma X$ and $0 < \epsilon < 1$ then \begin{align*}
        \Pr_{\omega\in\Sigma}\Big(s X(\omega) \not\in \big[ \mu(1-\epsilon),\mu(1+\epsilon) \big] \Big)\leq 2\,e^{\frac{-\mu\epsilon^2}{3}}
  \end{align*} 
\end{thm}

\begin{proof}[Proof of \ref{prop:9}]
  We will apply \ref{thm:CH} for $(\Sigma,\B,\nu) = \Omega^{[n]}_{p}$
  the Erd\"os-R\'enyi model and in the end choose the probability $p=p_n$. 

  For (i) fix $k$ vertices $A\subseteq [n]$ and
  take $m=n-k$ and $X(\Gamma)$ the 
  projection to
  the indicator vector
  of $N_\Gamma(A)$. In particular, we
  have $s\, X(\Gamma) = |N_\Gamma(A)|$
  and $\mu=m\,p^k$.
  Set
  \begin{align*}
  \epsilon & = \frac{\sqrt{3 \, \ln(m)\,(k+1)}}{  m^{\frac{1}{2}}p^{\frac{k}{2}}}.
  \end{align*}
  From the Chernoff-Hoeffding \ref{thm:CH} we get that
  if $\epsilon < 1$ then
  \begin{align*} 
     \Pr_{\Gamma\in\Omega^{[n]}_{p}}\Big(|N_\Gamma(A)| \not\in \big[m\,p^k(1-\epsilon),m\,p^k(1+\epsilon)\,\big]\,\Big) \leq 2 \,m^{-k-1}.
  \end{align*}
  Taking the sum over the $\frac{1}{k!}\,n(n-1)\cdots(m+1)\leq n^k$ choices for $A$ gives
  $$\Pr_{\Gamma\in\Omega^{[n]}_p}\Big(D_k(\Gamma)\not\subseteq \big[m\,p^k(1-\epsilon), m\,p^k(1+\epsilon)\,\big]\Big)\leq 2\,\Big(\frac{n}{m}\Big)^k m^{-1} .$$  
  Now set $p=p_n$ 
  which is middling of exponent $k$ so we have
  $\lim_{n\rightarrow \infty}\epsilon=0$.
  Since $\frac{n}{m} = \frac{1}{1-\frac{k}{n}} \leq 2$ for
  $n \geq 2k$ the right hand side goes to $0$ as $n$ grows.
  As a consequence (i) follows.  

  We now turn to (ii). 
  Fix $a$ and $b$ and let $A$ and $B$ be subsets of $[n]$ such that
  $A \cap B = \emptyset$, $|A| \geq a$ and $|B|\geq b$. 
  Let 
  $X$ be the projection to the indicator vector 
  of the edges connecting $A$ and $B$.  
  In particular $s\, X(\Gamma) = |E(K_{A,B}) \cap E(\Gamma)|$ 
  and $$\mu = |A| \, |B|\,p.$$
  Set 
  $$\epsilon :=\sqrt{\frac{3 (\ln(3^n) + n)}{abp}}.$$
  By \ref{thm:CH} if $\epsilon < 1$ then
  $$\Pr_{\Gamma \in \Omega_{p}^{[n]}}
  \Big( |E(K_{A,B}) \cap E(\Gamma)| \leq a\,b\,p\,(1-\epsilon)\Big)
  \leq 2\,3^{-n}\,e^{-n}.$$
  Since there are fewer than $3^n$ choices for $A$ and $B$ summing as before gives
  $$\Pr_{\Gamma \in \Omega^{[n]}_{p}} \Big(b_{a,b}(\Gamma) \leq a\,b\,p\,(1- \epsilon)\Big) \leq 2\,e^{-n}. $$

 Finally take $a=a_n$, $b=b_n$ and $p=p_n$ as in the hypotheses the
  assertion follows from $\lim_{n\rightarrow \infty} \frac{n}{a_nb_np_n} = 0$. 
\end{proof}

\section{Proof of \ref{lem:ki<=di} and \ref{thm:d1=k1}}

In the proofs we will make use of the pairing between
cochains and chains. For a number $i \geq 0$, a 
cochain $\chi \in C^i(\Cl(\Gamma);\KK)$ and a chain $\zeta \in C_i(\Cl(\Gamma);\KK)$ we write 
$(\chi,\zeta)$ to denote the evaluation of $\chi(\zeta)$. 
It is well known that this pairing is constant on
cohomology and homology classes and hence extends to a
paring between classes.

\subsection{$\kappa_\KK^i \leq \delta^i$}

\begin{proof}[Proof of \ref{lem:ki<=di}]
  Let $\Gamma$ be a graph and $A\in\Cl(\Gamma)[i]$ such that 
  $|N_\Gamma(A)|=\delta^i(\Gamma)$.
  
  \noindent {\sf Claim:} Aas every $A\in \Cl(\Gamma)$ with
  $|A|\leq i+1$ is a facet of an induced subcomplex of $\Cl(\Gamma)$
  isomorphic to the boundary of a cross polytope.  
  
  Let us first show that the claim implies the assertion.
  If the claim holds there is an integral homology cycle supported on the
  cross polytope which pairs to one with the cohomology class
  $$[\chi_{\{A\}}]\in \Ho_i(\Cl(\Gamma-N_\Gamma(A));\ZZ)$$ making it nontrivial.  
   Therefore with $A$ as above $\kappa_\KK^i(\Gamma) \leq |N_\Gamma(A)| = \delta^i(\Gamma)$.
  
   It remains to verify the claim:
   Let $C \subseteq V$, $|C| = k < 2i$ and $v \in V \setminus C$.
   Then for a graph $\Gamma$ on vertex set $V$ we have
   $N_\Gamma(C) - N_\Gamma(v) = N_\Gamma(C) - N_\Gamma(\{v \} 
   \cup C)$. In particular, since $p_n$ is middling of 
   exponent $2i \geq k+1$ we have aas
   for all $\epsilon > 0$ that
   \begin{align*}
       |N_\Gamma(C) - N_\Gamma(v)| & = & |N_\Gamma(C)| - |N_\Gamma(\{v\} \cup C)| \\
       & \overset{\ref{prop:9}(i)}{\geq} & 
       n\,p_n^k(1-\epsilon) - n\, p_n^{k+1} (1+\epsilon) \\
       & = & n\,p_n^k(1-\epsilon-p_n-p_n\epsilon).
   \end{align*}
   Again by the fact that $p_n$ is middling
   it follows that for all 
   $1 > \epsilon > 0$ we have aas that 
   $|N_\Gamma(C) - N_\Gamma(v)| \geq n\,p_n^k(1-\epsilon) > 0$.

   Now let $A=\{a_0,\ldots a_i\}$ be an $i$-face of $\Cl(\Gamma)$. 
   We construct 
   $b_0,\ldots, b_i$ such that the induced subcomplex of
   $\Cl(\Gamma)$ on 
   $A\cup\{b_0\ldots b_i\}$ is the boundary complex
   of an $(i+1)$-dimensional cross polytope.
   We proceed by induction and assume that for $0 \leq j \leq i$ we have already constructed $b_0,\ldots, b_{j-1}$. Set 
   $C = (A-\{a_j\})\cup\{b_0\ldots b_{j-1}\})$. Then by
   the arguments above for $k = |C| = i+j < 2i$ aas there is
   a vertex $b_j \in N_\Gamma((A-\{a_j\})\cup 
   \{b_0,\ldots, b_{j-1}\}) - N_\Gamma(a_j)$. 
   
   It follows that the subcomplex of $\Cl(\Gamma)$ induced on 
    the vertex set $\{a_0,\ldots, a_i,b_0 ,\ldots, b_i\}$ is the boundary
   complex of an $(i+1)$-dimensional cross polytope.
\end{proof}

By \ref{lem:ki<=di} in order to prove \ref{thm:d1=k1} it suffices to show that aas 
$\delta^1(\Gamma) \leq 
\kappa_{\FF_2}^1 (\Gamma)$. 
We will do this in the next sections.

\subsection{Main argument and higher homological dimension}	
For the remainder of the paper we are working over $\FF_2$ and with simplicial
complexes $X=\Cl(\Gamma)$. For an $i$-simplex 
$F \in X[i]$ we write $\chi_{\{F\}}$ (resp. $\zeta_{\{F\}}$) for its characteristic $i$-cochain (resp. $i$-chain). For a
collection $A = \{ F_1,\ldots, F_r\}$ we write 
$\chi_A$ for $\sum_{i=1}^r \chi_{\{F_i\}}$ and
$\zeta_{\{A\}}$ for $\sum_{i=1}^r \zeta_{\{F_i\}}$.
Using this notation we can make the
following identifications
$$C^i(X;\FF_2)={\{\chi_A\, : \, A\subseteq X[i])\}}$$ and 
$$C_i(X;\FF_2)={\{\zeta_A\, : \, A\subseteq X[i])\}}.$$

What remains is to argue that aas deleting vertices $C\subseteq X[0]$ with
even one vertex fewer than the minimum number of common neighbors of any
$i$-face will leave a complex with trivial
$i$\textsuperscript{th} homology or equivalently 
$i$\textsuperscript{th} cohomology.  Our suggestion is to
consider the dual norms on cohomology and homology of a cell complex over
$\FF_2$ given by 

\begin{align}
   \label{cohomnorm}
     \big\|\,\Ho^i( \Cl(\Gamma);\FF_2)\,\big\|_c=   \min\big\{|A|\,:\, 0 \neq [\chi_A]\in \Ho^i(\Cl(\Gamma);\FF_2)\big\}
\end{align}
and 
\begin{align}
  \label{homnorm}
 \big\|\,\Ho_i( & \Cl(\Gamma);  \FF_2)\,\big\|_h= \\  \nonumber & \min\big\{\max_j|A_j|\,:\,\{[\zeta_{\{A_1\}}],\ldots, [\zeta_{\{A_r\}}]\}\, \hbox{ spans }\Ho_i(\Cl(\Gamma);\FF_2)\,\big\}.
\end{align}



The following conjecture lays out a plan for proving 
\ref{con:main} in general. We will verify this conjecture for
$i = 1$. 

\begin{Conjecture}
  \label{con:delicate}
  For all $i\in\NN$ and $\epsilon>0$ there exist
  $r_i,\ell_i,w_i \geq 0$ such that for all middling 
  $p:\NN\rightarrow [0,1]$ with exponent $w_i$ we have that 
  $$\lim_{n \rightarrow \infty} \Pr_{\Omega^{[n]}_{p_n}} ( \Gamma \textrm{ satisfies } (\ast) ) = 1,$$ where $(\ast)$ is the
  following condition:
  
  \smallskip
  
  \noindent $(\ast)$ For all $C\subseteq [n]$ with $|C|\leq n\,p_n^{i+1}(2-\epsilon)$ the following properties hold
  \begin{enumerate}
  \item $\|\Ho_i(\Cl(\Gamma-C);\FF_2)\|_h\leq\ell_i,$
  \item either 
    \begin{itemize}
        \item there is $v\in [n], N_\Gamma (v)\subseteq C$ or
        \item $\|\Ho^i(\Cl(\Gamma-C);\FF_2)\|_c\geq n^{i+1}p_n^{r_i}.$
    \end{itemize}
  \end{enumerate}
\end{Conjecture}

The next two definitions and two lemmas will be used to relate
  \ref{con:delicate} to \ref{con:main}.

\begin{Definition} ~ 
   \begin{itemize}
       \item[(i)] 
   Let $X$ be a simplicial complex and 
   $\gamma\in \Ho^i(X;\FF_2)$. A set
   $\K=\{\zeta_{A_1},\ldots, \zeta_{A_r}\}\subseteq Z_i(X;\FF_2)$ of disjoint cycles
   with $|A_1|,\ldots, |A_r|\leq \ell$, each pairing to $1$ with $\gamma$ is called $\ell$-adapted to $\gamma$. 
   We set $$k^i_\ell(X) = \min_{0 \neq \gamma\in\Ho^i(X;\FF_2)} \max\big\{\,|\K|\ :\ \K\text{ is $\ell$-adapted to }\gamma\big\},$$ where we set $k^i_\ell(X)=\infty$ if $\Ho^i(X;\FF_2)=0$.  
   \item[(ii)] For a graph $\Gamma$ we set 
   $$k^i_\ell(\Gamma) = \min_{\genfrac{}{}{0pt}{}{C\subseteq V(\Gamma)}{|C|<\delta^i(\Gamma)}} k^i_\ell(\Cl(\Gamma-C)).$$
   \end{itemize}
\end{Definition}

\begin{lem}
  \label{many}
  If \ref{con:delicate} holds and $p:\NN\rightarrow [0,1]$ is middling
  of exponent $i+1$ and $w_i$ then for every $t$ one has
  $$\lim_{n\rightarrow \infty}\Pr_{\Gamma\in\Omega^{[n]}_{p_n}}\Big(k^i_{\ell_i}(\Gamma) \geq p_n^{-t}\Big)=1.$$
\end{lem}
\begin{proof}
  Assume for contradiction that for some $\epsilon >0$ we have
  $$\Pr_{\Gamma\in \Omega_{p_n}^{[n]}}(k_{\ell_i}^i(\Gamma) < p_n^{-t}) \leq 1-\epsilon$$
  for infinitely many $n$ and assume without loss of generality that 
  $t\leq i+1$.
  For each such $n$ and each graph $\Gamma$ with 
  $V(\Gamma) = [n]$ for which 
 $$k_{\ell_i}^i(\Gamma) < p_n^{-t}$$ choose
  a maximal $\K$ that is $\ell_i$-adapted to some
  $$0 \neq \gamma\in \Ho^i(\Cl(\Gamma-C);\FF_2)$$ with $|C|<\delta^i(\Gamma)$. Thus we have $|\K| =k_{\ell_i}^i(\Gamma)< p_n^{-t}$. 
  Now write 
  $$D=\bigcup_{\zeta_F\in\K} \, \bigcup_{A\in F}A\subseteq V(\Gamma)$$
  for the set of all the vertices involved in $\K$.  
  Simple counting shows $|D|\leq (i+1)\,\ell_i\,p_n^{-t}$. By \ref{prop:9}(i)
  with exponent $i+1$ we
  have aas that $\delta^i(\Gamma) \leq \,n\,p_n^{i+1}$ and hence
  $|C\cup D|<np_n^{i+1}(2-\frac{1}{2})$. By \ref{con:delicate}(1) for
  $\epsilon = \frac{1}{2}$ it follows that
  $$\|\Ho_i(\Cl(\Gamma-(C \cup D));\FF_2)\|_h\leq \ell_i.$$  
  We denote by $$j:\Cl(\Gamma-(C \cup D))\rightarrow \Cl(\Gamma-C)$$ the natural injection and by $j^\ast$ the map induced by
  $j$ in cohomology. Then if
  $$0\neq j^*(\gamma)\in\Ho^i(\Cl(\Gamma-(C\cup D));\FF_2)$$ 
  then there must be some
  $$\zeta_F\in Z_i(\Cl(\Gamma-(C \cup D));\FF_2)$$ with $|F|\leq \ell_i$ and
  $([\zeta_F],j^*(\gamma))=1$ contradicting the maximality of $\K$.
  Finally consider the case $j^*(\gamma)=0$. Then 
  $\gamma = [\chi_E] + [\partial \chi_{E'}] = [\chi_E]$ for an $i$-cocycle $\chi_E$ and an $(i-1)$-chain $\chi_{E'}$
  such that  
  $E$ is a is set of $i$-simplices
  all containing a vertex from $D$ and $E'$ is a set of $(i+1)$-simplices supported on 
  $V(\Gamma) - (C \cup D)$.
  It follows that 
  $||H^i(\Gamma-(C-D);\FF_2)||_c$
  is bounded from above by the number of $i$-simplices in $\Gamma-C$ with a least one vertex in $D$. A rough counting argument then shows that 
  $||H^i(\Gamma-(C-D);\FF_2)||_c
  \leq |D|\,\binom{n} {i} \leq (i+1)\,\ell_i\,p_n^{-t}\,n^i< n^{i+1}p_n^{r_i}$. But this 
  contradicts \ref{con:delicate}(2).  
\end{proof}

\begin{lem}\label{pair}
      For every graph $\Gamma$ and numbers $i  \geq 0$ and $\ell \geq 1$  
      at least one of the following holds:
      \begin{itemize}
      \item[(i)] There is  $A \subseteq V(\Gamma)$ with
        $|A|\leq\ell$ such that
          \begin{eqnarray*}
       \delta^i(\Gamma)& \geq    & \frac{1}{2\ell}\,|N_\Gamma(A)|\,k^i_\ell(\Gamma)
          \end{eqnarray*}
          \item[(ii)] There are disjoint $A,B \subseteq V(\Gamma)$ with 
            $|B|=|A|\leq \ell$ such that
          \begin{eqnarray*}
          \delta^i(\Gamma) & \geq & \frac{|N_\Gamma (A)|^2}{|N_\Gamma (A\cup B)|}.
          \end{eqnarray*}
          \item[(iii)] $\kappa_{\FF_2}^i(\Gamma) \geq \delta^i(\Gamma)$.
      \end{itemize}
      \end{lem}
    \begin{proof} It suffices to show that (i) or (ii) holds in case $\kappa_{\FF_2}^i(\Gamma) < \delta^i(\Gamma)$. 
      Choose $C\subseteq V(\Gamma)$ with $|C|=\kappa_{\FF_2}^i(\Gamma)$ and
      $\K$ a set of $k=k_\ell^i(\Gamma)$ cycles $\ell$-adapted to
      $0 \neq \gamma=[\chi_E]\in \Ho^i(\Gamma;\FF_2)$. 
    Choose $\ell'\leq\ell$
    such that the set $\K'$ of $\ell'$-cycles in $\K$ 
    satisfies $|\K'|\geq \frac{k}{\ell}$. Set $U=\cup_{A\in \K'}
    N_\Gamma (A)$.
    Note that $U\subseteq C$ since any $u\in U-C$ would be a cone point for
    a cycle in $\K$ contradicting its nontrivial pairing with $\gamma$.  
        Set \begin{eqnarray*}
      a & = & \min_{A\in \K'} |N_\Gamma (A)| \\
      b & = & \max_{A,B\in\K'}|N_\Gamma(A\cup B)|.
    \end{eqnarray*} 
        Then $$\delta^i(\Gamma)> \kappa_{\FF_2}^i(\Gamma) =   |C| \geq |U|\geq\min\Big\{\frac{ak}{2\ell},\frac{a^2}{2b}\Big\}$$
        with the last inequality holding for the union of any collection of
        at least $\frac{k}{\ell}$ sets each with size at least $a$ and
        pairwise intersections of size at most $b$.  The lemma follows.  
              \end{proof}

    \begin{prop}
    \label{prop:conthm}
  \ref{con:delicate}(i) implies \ref{con:main}(i) with exponent $W_i= \max\{ w_i, 2\ell_i, 2i\}$. 
  \end{prop}
\begin{proof}
  We already know from \ref{lem:ki<=di} that in order to prove
  \ref{con:main} for any $W_i\geq 2i$ it suffices to show that aas 
  $\delta^i(\Gamma) \leq \kappa_{\FF_2}^i (\Gamma)$.

  Assume for contradiction that for some $\epsilon >0$ there are 
  infinitely many $n$ such that
  $$\Pr_{\Gamma\in\Omega^{[n]}_{p_n}}\Big(\kappa^i_{\FF_2}(\Gamma)\geq\delta^i(\Gamma)\Big) \leq 1-\epsilon.$$
  
  By \ref{prop:9}(i) and $W_i\geq i+1$ one has aas
  \begin{eqnarray} \label{eq:0}
  \delta^i(\Gamma) & \leq & n\,p_n^{i+1}\,(1+\epsilon).
  \end{eqnarray}
  
  For a graph $\Gamma$ with $\delta^i(\Gamma) > \kappa_{\FF_2}^i(\Gamma)$
  and $\ell = \ell_i$ let $A$ be as in \ref{pair}(i) or $A$ and $B$
  as in \ref{pair}(ii) with $|A|=a\leq\ell$.  
  By \ref{prop:9}(i) and $W_i\geq \ell_i$ we
  have aas that 
  \begin{eqnarray}
     \label{eq:a}
  |N_\Gamma (A)|& \geq &  n\,p_n^{a}\,(1-\epsilon)
  \end{eqnarray} and similarly since $W_i\geq 2\ell_i$ aas
  \begin{eqnarray}
     \label{eq:b}
   |N_\Gamma (A\cup B)|&\leq & n\,p_n^{2a}\,(1+\epsilon).
  \end{eqnarray}
  
  In the case \ref{pair}(i) one has aas the contradiction with the last
  inequality using that $W_i\geq w_i$
  \begin{eqnarray*} 
    n & > &    n \, p_n^{i+1} \,(1+\epsilon) \\
    & \overset{\eqref{eq:0}}{\geq} &  \delta^i(\Gamma) \\ 
  &\overset{\ref{pair}(i)}{\geq} &
  \frac{1}{2 \ell_i} |N_\Gamma(A)| \, k_{\ell_i}^i(\Gamma)\\
  &\overset{\eqref{eq:a}}{\geq}& \frac{1}{2\ell_i} \,n\, p_n^a \, k_{\ell_i}^i(\Gamma)(1-\epsilon) \\ &\geq &
  \frac{1}{2\ell_i}\,n\, p_n^{\ell_i} \, k_{\ell_i}^i(\Gamma)(1-\epsilon) \\
  & \overset{\ref{many}}{\geq} & n.
  \end{eqnarray*}
  
  To complete the proof in the case \ref{pair}(ii) one has aas the contradiction
  \begin{eqnarray*}
    \frac{n}{2} & > & n\, p_n^{i+1}\,(1+\epsilon) \\
    & \overset{\eqref{eq:0}}{\geq} & \delta^i(\Gamma) \\ & \overset{\ref{pair}(ii)}{\geq} & 
    \frac{|N_\Gamma(A)|^2}{|N_\Gamma(A\cup B)|} \\
    & \overset{\eqref{eq:a},\eqref{eq:b}}{\geq} & 
    \frac{n^2\, p_n^{2a}\,(1-\epsilon)^2}{
    n\,p_n^{2a} \,(1+\epsilon)} \\
    & = & \frac{(1-\epsilon)^2}{1+\epsilon} n.
    \end{eqnarray*}
  
  \end{proof}

In the sections \ref{sec:cor41.1} and \ref{sec:cor41.2}
we will prove \ref{con:delicate}
for $i =1$ and hence \ref{thm:d1=k1}.  
The rest of this section is an aside on a 
feature of middling probabilities which we consider 
as a promising tool for handling the case $i>1$.

We observe that the middling property of probabilities
$p : \NN \rightarrow [0,1]$ is inherited by links.
More precisely, for a graph $\Gamma = ([n],E(\Gamma))$ and a subset $A \subseteq [n]$ of its vertices, the (graph theoretic) link of $A$ in $\Gamma$ is $\Gamma|_{N_\Gamma(A)}$.
It is easily seen that when $A$ is a clique then $\Cl(\Gamma|_{N_\Gamma(A)})$ is
isomorphic to the simplicial link of $A$ in $\Cl(\Gamma)$. 

 Let $p:\NN\rightarrow [0,1]$ be middling and $i$ an non-negative integer. 
 Define $q:\NN\rightarrow [0,1]$ by 
 $q_{m,i}=\min\{\,p_n\,:\,np^i_n\leq m\,\}$.
 Define two probability measures on the set of graphs $\Gamma$ with vertex set $V(\Gamma)$ contained in $[n]$:
 $$\mu_{n,i}(\Gamma)=\binom{n}{i}^{-1}\sum_{A\in\binom{[n]}{i}}\Pr_{\Gamma'\in\Omega^{[n]}_{p_n}}(\Gamma'|_{N_\Gamma'(A)}=\Gamma),$$
  $$\nu_{n,i}(\Gamma)=p_n^{i\,|V(\Gamma)|}\,(1-p_n^i)^{|V(\Gamma)|}\,\Pr_{\Omega_{q_{|V(\Gamma)|}}^{V(\Gamma)}}(\Gamma).$$
The total variation distance between these approaches zero.  

\begin{lem}
  For every $p$ middling and positive $i$ one has
  $$\lim_{n\rightarrow \infty}\sum_{\Gamma=([n],E)}\big|\mu_{n,i}(\Gamma)-\nu_{n,i}(\Gamma)\big|=0.$$
  \end{lem}

\subsection{Large cocycles (\ref{con:delicate}(1))} \label{sec:cor41.1}

For a graph $\Gamma$ we say that a pair $(C,E)$ 
satisfies (S) if $C\subseteq V(\Gamma)$,there is no $v \in V(\gamma)$ such that $N_\Gamma(v)\subseteq C$ and $E\subseteq E(\Gamma)$ defines a class $0\neq[\chi_E]\in H^1(\Cl(\Gamma-C);\FF_2)$ with $|E|$ is minimal among the representatives of the class.

The following lemma will be useful in various situations.

\begin{lem}
  \label{lem:triangle}
  Let $\Gamma$ be a graph and assume (S) holds for $(C,E)$. 
  Then every edge in $\Gamma-C$ is contained
  in a triangle in $\Gamma-C$. 
  Conversely, if
  $T$ is a triangle in $\Gamma$ 
  then exactly one of the following situations holds:
  \begin{itemize}
    \item[(i)] $T$ has no edges in $E$, 
    \item[(ii)] $T$ has two edges in $E$, 
    \item[(iii)] $T$ has one edge in $E$ and the opposite vertex in $C$  
  \end{itemize}
\end{lem}
\begin{proof} 
  The first claim follows immediately from
  the definition of $\delta^1$ and the
  fact that $|C| < \delta^1(\Gamma)$. 
  
  If $T$ has three edges in $E$ or
  if $T$ has one edge in $T$ and the
  opposite vertex in $V(\Gamma)-C$ then 
  the cochain supported on $T$ has a non-zero coefficient in the
  coboundary of $E$ in $\Gamma-C$. 
\end{proof}

Set  
$$f(\Gamma)=\min\Big\{\max_{v\in V(\Gamma)-C} \big|N_E (v)\big|\ :\ (C,E)\text{ satisfies (S) for }\Gamma\Big\}.$$  

\begin{lem}  
  \label{lem:dense}
  If $p$ is middling of exponent $3$ then $$\lim_{n\rightarrow\infty}\Pr_{\Gamma\in\Omega^{[n]}_{p_n}}\Big(f(\Gamma) \geq \frac{1}{4}\,n\,p_n^2\Big)=1.$$
\end{lem} 
\begin{proof} Assume that $\Gamma$ is a 
  graph and $(C,E)$ satisfy (S). Since
  $[\chi_E] \neq 0$ we have $|E| >0$ and we can choose $e=\{u,v\}\in E$.  By
  \ref{lem:triangle} there is $w \in
  V(\Gamma) -C$ such that
  $\{w,u,v\}$ is a triangle in $\Gamma-C$.
  Again by \ref{lem:triangle} it follows
  that exactly one of the edges $\{w,u\}$
  and $\{w,v\}$ lies in $E$. We may assume
  $e' = \{w,u\}$ lies in $E$.
  
  We have 
  \begin{flalign}
 \nonumber   | N_E(u)|+ |N_E& (v)|+|N_E(w)|  \\
 \label{eq:first}  &\geq  |(N_\Gamma (e)\cup N_\Gamma (e')) \setminus C| \\
    & \geq  |N_\Gamma (e) \cup N_\Gamma (e')| - |C| \\
 \nonumber   & \overset{|N_\Gamma (e')| \geq |C|}{\geq}  |N_\Gamma (e) \cup N_\Gamma (e')| - |N_\Gamma (e')| \\
\nonumber	  & =  |N_\Gamma (e) - (N_\Gamma (e) \cap N_\Gamma (e'))| \\
\nonumber	  & =  |N_\Gamma (e) - N_\Gamma (e\cup e') | \\
\label{eq:3.1}	  & \geq   n\,p_n^2\,(1-\epsilon - p_n -\epsilon p_n)\\
\label{eq:l2}	  & {\geq} \, \frac{3}{4}\,n\,p_n^2.
  \end{flalign}
The first inequality \eqref{eq:first} follows from \ref{lem:triangle}. Note that if $x\in N_\Gamma (e)$ then $\{x,u,v\}$ is a triangle in $\Gamma-C$ with $e=\{u,v\}\in E$ so exactly one of $\{x,u\}$ and $\{x,v\}$ is also in $E$. A similar reasoning applies to $e'$.  The penultimate \eqref{eq:3.1} inequality holds aas by \ref{prop:9}(i) with exponent $3$. The last inequality \eqref{eq:l2} holds for sufficiently large $n$ because $p_n$ approach $0$.  

If follows that at least one of the sets
$N_E(u)$, $N_E(v)$ or $N_E(w)$ has size
$\geq \frac{1}{4} \, n \, p_n^2$. This implies the assertion.
\end{proof}
   
\begin{lem}   
  If $p$ is middling of exponent $4$, $0 < \epsilon < 1$, $C\subseteq [n]$
  with $|C|\leq n\,p_n^{i+1}(2-\epsilon)$ and $r_1 > 4$ then
  $\|\Ho^1(\Cl(\Gamma-C);\FF_2)\|_c\geq n^{i+1}p_n^{r_1}$.
In particular, 
\ref{con:delicate}(2) holds if $i = 1$, $r_1>4$ and $w_1=4$.  \end{lem}
\begin{proof} Consider a graph $\Gamma$ with $V(\Gamma)=[n]$.  For contradiction
  assume there is a pair $(C,E)$ in which $C\subseteq [n]$ with no
  $N_\Gamma(v)\subseteq C$ and $|C|\leq \frac{3}{2}np_n^2$ as well as some
  $0\not=[\chi_E]\in H^1(\Cl(\Gamma-C);\FF_2)$ with $|E|<n^2p_n^4\frac{1}{3}$.
  It suffices to show that the probability that $\Gamma\in\Omega^{[n]}_{p_n}$
  has such a pair approaches $0$ as $n$ grows.  

  If $\Gamma$ has such a pair $(C,E)$ choose one also satisfying (S).
  Note that if $v \in [n] -C$ and $A = N_E(v)$ and
  $B = N_{\Gamma-C}(v) - N_E(v)$ then by \ref{lem:triangle} 
  $$E\cap E(K_{A,B}) =E(\Gamma) \cap E(K_{A,B}).$$ 
  Hence the assertion follows if we show
  that there is a choice of $v \in V(\Gamma)-C$ such that
  $|E(\Gamma) \cap K_{A,B}| \geq \frac{1}{12} \,n^2\,p_n^4$ aas.
  
  This will follow from \ref{prop:9}(ii) 
  once we have shown that $|A|$ and $|B|$
  are ``large.'' Note that since $|E|$ is minimal we have 
  $$|E| \leq |E \cup \partial^1(v)| =
  |E|+|B|-|A|.$$
  It follows that $|B| \geq |A|$. 
  By \ref{lem:dense} $f(\Gamma) \geq\frac{1}{4}\ n\,p_n^2$ aas. 
  Hence we can
  choose $v$ such that $|N_E(v)| = |A| 
  \geq \frac{1}{4}\, n\, p_n^2$.
  Also by \ref{prop:9}(i) and the exponent $4\geq 1$ we
  have that for any $\epsilon > 0$ aas
  \begin{align*} 
   2|B| & \geq & |A|+|B| \\ & = & |N_{\Gamma-C}(v)| \\ & \geq & |N_\Gamma(v)| - |C| \\ & \overset{\ref{prop:9}(i)}{\geq} &
   n\,p_n\,(1-\epsilon) + |C| \\
   & \geq  & n\, p_n\, (1+\epsilon)
   \end{align*}  
   Now by \ref{prop:9} (ii) for
   $\alpha_1 = 2$ and $\alpha_0=\beta_0=\beta_1 = 1$, and exponent
   $4\geq \frac{1+2+1}{1+1-1}=4$ we get that for any 
   $\epsilon > 0$ there is aas
   $$|E|\geq|E(\Gamma) \cap E(K_{A,B})| \geq 
   \frac{1}{2}\,n^2\,p_n^4\,(1-\epsilon).$$ 
\end{proof}

\subsection{Many small cycles (\ref{con:delicate}(2))} \label{sec:cor41.2}
  By \ref{prop:conthm} the next lemma completes the proof of \ref{thm:d1=k1}. 
In it proof we write $\tau(\Gamma)$ for the minimum 
cardinality of a set $D\subseteq V(\Gamma)$ for which $\Gamma-D$ has at least two pairs of vertices with distance greater than $2$.

\begin{lem}  
  \ref{con:delicate}(1) holds for $i=1$, $\ell_1=5$ and $W_1=3$.  \end{lem}
\begin{proof}
  Note that it suffices to show that if $\Gamma\in\Omega^{[n]}_{p_n}$ then aas
  $\tau(\Gamma)>\frac{3}{2}np_n^2$ since if a graph has at most one pair of
  vertices of distance more than $2$ then every cycle of
  length at least $6$ has a chord and hence if $|C|<\tau(\Gamma)$ then
  $\|H_1(\Gamma-C);\FF_2)\|_h\leq 5$.
Thus it suffices to show that aas for any two 
pairs $x,y$ and $x',y'$ of vertices in $V(\Gamma)$ we have
$|N_\Gamma(\{x,y\}) \cup N_\Gamma(\{x',y'\})| > \frac{3}{2}np_n^2$.
  We apply \ref{prop:9}(i) with exponent $2$ to show that 
  $$|N_\Gamma(\{x,y\})|, |N_\Gamma(\{x',y'\}| \geq np_n^2(1-\epsilon)$$
  aas and with exponent $3$ to show that
  $$|N_\Gamma(\{x,x',y,y'\})| \leq np_n^3 (1+\epsilon)$$
  aas. The result follows. 
  \end{proof}

\end{document}